\documentclass[11pt]{amsart}
\usepackage[frame,ps,matrix,arrow,curve,rotate]{xy}
\usepackage{xspace}

\usepackage{setspace}
\newtheorem{theorem}{Theorem}[section]
\newtheorem{proposition}[theorem]{Proposition}
\newtheorem{conjecture}[theorem]{Conjecture}

\newtheorem{lemma}[theorem]{Lemma}
\newtheorem{corollary}[theorem]{Corollary}

\theoremstyle{remark}
\newtheorem{remark}[theorem]{\bf Remark}

\theoremstyle{definition}

\newcommand{\PP}{\mathbb{P}}
\newcommand{\HH}{{\mathcal H}}

\newcommand{\cO}{{\mathcal O}}
\newcommand{\cU}{{\mathcal U}}

\newcommand{\cI}{{\mathcal I}}

\newcommand{\Hilb}{{\rm Hilb}}

\newcommand{\M}{\mathcal M}

\newcommand{\Gr}{\mathop{\rm Gr}}

\begin{document}

\title{Spaces of Rational Curves in Complete Intersections}
\date{}
\author{Roya Beheshti}
\address{Department of Mathematics, Washington University in St. Louis,
St. Louis, Missouri, 63130}
\email{beheshti@wustl.edu}
\author{N.~Mohan Kumar}
%\address{Department of Mathematics, Washington University in St. Louis,
%St. Louis, Missouri, 63130}
\email{kumar@wustl.edu}
\urladdr{http://www.math.wustl.edu/~kumar}

\begin{abstract} We prove that the space of smooth rational curves of
  degree 
$e$ in a general complete intersection of multidegree $(d_1, \dots, d_m)$ 
in $\PP^n$ is irreducible of the expected dimension if $\sum_{i=1}^m
d_i <\frac{2n}{3}$  
and $n$ is large enough. This generalizes the results of Harris, Roth and
Starr \cite{hrs}, and is achieved by proving that the space of conics 
passing through any point of a general complete intersection has constant dimension if 
$\sum_{i=1}^m d_i$ is small compared to $n$. 
\end{abstract}

\maketitle

\section{Introduction}

Throughout this paper, we work over the field of complex numbers. For
a smooth projective variety 
$X \subset \PP^n$ and an integer 
$e \geq 1$, we denote by  ${\Hilb}_{et+1}(X)$ the Hilbert scheme
parametrizing subschemes of  
$X$ with Hilbert polynomial $et+1$, and we denote by 
$R_e(X) \subset {\rm{Hilb}}_{et+1}(X)$ the open subscheme parametrizing 
smooth rational curves of degree $e$ in $X$. If $X=\PP^n$, then $R_e(X)$ is a 
smooth, irreducible, rational variety of dimension $(e+1)(n+1)-4$. But
already in the case of hypersurfaces  
in $\PP^n$,
there are many basic questions concerning the geometry of $R_e(X)$ which are still open.
In this paper we address and discuss some of these questions, focusing in particular on the 
dimension and irreducibility of $R_e(X)$,  when $X$ is a general complete intersection in $\PP^n$.

To study the space of smooth rational curves in $X$, we consider the
Kontsevich moduli space of stable maps $\overline{\M}_{0,0}(X,e)$
which compactifies  
$R_e(X)$ by allowing smooth rational curves to degenerate to morphisms
from nodal curves.  These have certain advantages over the Hilbert
Schemes for the problems studied here.
%The advantage of $\overline{\M}_{0,0}(X,e)$ over 
%the Hilbert scheme 
%is that its deformation theory is simple and 
%it has a functorial property. 
We refer the reader to 
\cite{bm, ds, fp, hrs} for detailed discussions of these moduli spaces
and the comparison  between them.

For every smooth hypersurface $X \subset \PP^n$ of degree $d$, the dimension of every irreducible 
component of $\overline{\M}_{0,0}(X,e)$ is at least $e(n+1-d)+n-4$, and if $d \leq n-1$, then there is at least one irreducible component whose dimension is equal to $e(n+1-d)+n-4$ (see 
Sections 2.1 and 6). The number $e(n+1-d)+n-4$ is referred to as {\em the expected dimension} of $\overline{\M}_{0,0}(X,e)$.
If $X$ is an arbitrary smooth hypersurface, $\overline{\M}_{0,0}(X)$ (or even $R_e(X)$) can be reducible and its dimension can be larger than expected 
(see \cite{izzet}, Section 1). By a result of 
Harris, Roth, and Starr \cite{hrs}, if $d< \frac{n+1}{2}$ and $X$ is a {\it{general}} hypersurface of degree $d$ in $\PP^n$, then 
for every $e \geq 1$, $\overline{\M}_{0,0}(X,e)$ is integral of the 
expected dimension and has only local complete intersection singularities. In this paper, we generalize this result to higher degree 
hypersurfaces.

Let $\overline{\M}_{0,1}(X,e)$ denote the moduli space of 1-pointed stable maps of degree $e$ to $X$. 
In order to obtain the above mentioned result, Harris, Roth, and Starr show that if $d < \frac{n+1}{2}$ and $X$ is a general hypersurface of degree 
$d$ in $\PP^n$, 
then the  evaluation morphism
$$
ev: \overline{\M}_{0,1}(X,e) \to X
$$
is flat of relative dimension $e(n+1-d)-2$ for every $e \geq 1$ (\cite[Theorem 2.1 and Corollary 5.6]{hrs}). It is conjectured that the same holds for any $d \leq n-1$:

\begin{conjecture}[Coskun-Harris-Starr \cite{izzet}]
Let $X$ be a general hypersurface of degree $d \leq n-1$ in $\PP^n$. Then the evaluation morphism 
$$ev: \overline{\M}_{0,1}(X,e) \to X$$
is flat of relative dimension $e(n+1-d)-2$ for every $e \geq 1$.
\end{conjecture}

The above conjecture would imply the following: 

\begin{conjecture}[Coskun-Harris-Starr \cite{izzet}]\label{conjA} 
If $X$ is a general hypersurface of degree $d \leq n-1$ in $\PP^n$, then for every $e \geq 1$, $\overline{\M}_{0,0}(X,e)$ has the expected dimension $e(n+1-d)+n-4$. 
\end{conjecture}

Coskun and Starr \cite{izzet} show that Conjecture \ref{conjA} holds for $ d< \frac{n+4}{2}$. When $d=n-1$ and $e \geq 2$,  
$\overline{\M}_{0,0}(X,e)$ is reducible for the following reason. By Lemma \ref{expected}, $\overline{\M}_{0,0}(X,e)$ has at least one irreducible component of dimension $2e+n-4$ whose general point parametrizes 
an embedded smooth rational curve of degree $e$ in $X$. On the other hand, the space of 
lines in $X$ has dimension at least $n-2$, and therefore, the space of degree $e$ covers of lines in $X$ has dimension $ \geq (n-2)+(2e-2)$, thus $\overline{\M}_{0,0}(X,e)$ 
has at least 2 irreducible components. It is expected that if $X$ is general, then $\overline{\M}_{0,0}(X,e)$ 
is irreducible  when $d \leq n-2$, and $R_e(X)$ is irreducible when $d \leq n-1$ (see \cite[Conjecture 1.3]{izzet}).

In this paper, we show: 

\begin{theorem}\label{main}
Suppose that $X$ is a general hypersurface of degree $d$ in $\PP^n$. If $$n -1 < {{n-d} \choose {2}},$$ then 
the evaluation morphism $ev: \overline{\M}_{0,1}(X,2) \to X$ 
is flat of relative dimension $2n-2d$.  
\end{theorem}

A smooth rational curve in $X$ is called {\em free} if its normal bundle in $X$  is globally generated. 
The proof of the above theorem is based on an analysis of the space of non-free conics in $X$ passing through an arbitrary point of $X$.
It seems quite plausible that the same approach can be applied  to the case of cubics or other higher degree rational curves to prove the flatness of 
$ev$ when $d$ and $n$ satisfy the inequality of the theorem, 
but we have not carried out all the details, and we restrict the discussion here to the case of conics.

Theorem \ref{main} along with the results of \cite{hrs} gives  the following:

\begin{theorem}\label{maincor}
If $X \subset \PP^n$ is a general hypersurface of degree $d < \frac{2n}{3}, n \geq 20$, 
then for every $e \geq 1$, 
\begin{itemize}
\item[(a)]The evaluation morphism $\overline{\M}_{0,1}(X,e) \to X$ is flat and of relative dimension $e(n+1-d)-2$. 
\item[(b)] $\overline{\M}_{0,0}(X,e)$ is an irreducible local complete intersection stack of expected dimension $e(n+1-d)+(n-4)$. 
%\item[(c)] $\overline{\M}_{0,0}(X,e)$ is irreducible. 
\end{itemize}
\end{theorem}

The above results can be generalized to the case of general complete intersections.  
Suppose that $X \subset \PP^n$ is a general complete intersection of multidegree $(d_1, \dots, d_m)$ and set $d = d_1 + \dots + d_m$. If  
$n -m < {{n-d} \choose {2}},$ then the evaluation morphism $ev: \overline{\M}_{0,1}(X,2) \to X$ is flat of relative dimension $2(n-d)$, 
and if $d$ and $n$ are in the range of Theorem \ref{maincor},  
then $\overline{\M}_{0,0}(X,e)$ is an irreducible, local complete intersection stack of expected dimension $e(n+1-d)+n-3-m$ (Theorem \ref{last}).

Along the way to proving Theorem \ref{main}, we obtain the following result on the space of non-free lines in general hypersurfaces.

\begin{theorem}\label{lines}
Let $X$ be a general hypersurface of degree $d$ in $\PP^n$, and let $p$ be an arbitrary point of $X$. For $1 \leq k \leq d-1$, set 
$$
a_k = \min \left \{ a \geq 0 \; | \; {{a+k+2} \choose {k+1}} \geq n \right \} .$$  
Then the family of lines $l$ in 
$X$ passing through $p$ with $h^1(l, N_{l/X}(-1)) \geq k$ has dimension $\leq a_k$.   
\end{theorem}

In fact, we can modify the proof of the above theorem to say more in special cases. For example, 
it follows from the proof of the theorem that if $n \leq 5$, then the space of non-free lines through any point of $X$ is at most 
zero dimensional. 
The proof shows that 
if there is a 1-parameter family of non-free lines in $X$ through $p$ parametrized by $C \subset \PP^{n-1}$, 
then $n \geq 2 + 2 \dim$ Linear Span($C$) (see Proposition \ref{linearspan}). Of course, a general hypersurface of degree $\geq 3$ in $\PP^n$, $n \leq 5$, does not contain any 
2-plane, so the dimension of the linear span of $C$ is at least 2. Note that for a general hypersurface $X$ of degree $3 \leq d \leq n-1$, the 
non-free lines in $X$ sweep out a divisor in $X$ (Proposition \ref{sweep}). 

\bigskip

\paragraph{\bf{Acknowledgements}} 
The first named author is grateful to Izzet Coskun for many  helpful conversations and for 
pointing out to us how to make our original proof of Theorem \ref{main} shorter. She also thanks Matt Deland and Jason Starr for useful discussions.
%%%%%%%%%%%%%%%%%%%%%%%%%%%%%%%%%%%%%%%%%
\section{Background and Summary}
%%%%%%%%%%%%%%%%%%%%%%%%%%%%%%%%%%%%%%%%%

\subsection{Preliminaries} 

Fix positive integers $d_1 \leq \dots \leq d_m$, and set $d := d_1 + \dots + d_m$. Let $X$ be a smooth complete intersection of multidegree $(d_1,\dots,d_m)$ in 
$\PP^n$. The Kontsevich moduli space $\overline{\M}_{0,r}(X,e)$ parametrizes 
isomorphism classes of tuples $(C, q_1, \dots, q_r, f)$ where 
\begin{enumerate}
\item $C$ is a proper, connected, at worst nodal curve of arithmetic genus 0. 
\item $q_1, \dots, q_r$ are distinct smooth points of $C$. 
\item $f: C \to X$ is a map of degree $e$ in the sense that the pull
  back of the hyperplane section line bundle of $X$ has total degree
  $e$ on $C$, which further satisfies the following stability
  condition:  any irreducible component of $C$ which is mapped to a  
point by $f$ has at least 3 points which are either marked or nodes. 
\end{enumerate}

\noindent The tuples $(C, q_1, \dots, q_r, f)$ and $(C',q_1',\dots, q_r', f')$ are isomorphic if there is an isomorphism 
$g:C \to C'$ taking $q_i$ to $q_i'$, with $f'\circ g = f$. The moduli space $\overline{\M}_{0,r}(X,e)$ is a proper Deligne-Mumford stack, and the corresponding coarse moduli space 
$\overline{M}_{0,r}(X,e)$ is  a projective scheme. There is an evaluation morphism 
$$ev: \overline{\M}_{0,r}(X,e) \to X^r$$
sending a datum $(C, q_1, \dots, q_r, f)$ to $(f(q_1),\dots, f(q_r))$.  We refer to \cite{bm} and \cite{fp} 
for constructions and basic properties of these moduli spaces. 

The space of first order deformations of the map $f$ with $(C, q_1,\dots, q_r)$ fixed 
can be identified with $H^0(C,f^*T_X)$. If $H^1(C,f^*T_X)=0$ at  a point $(C, q_1,\dots, q_r, f)$, then 
$f$ is unobstructed and the moduli stack is smooth at that point. 
In particular, $\overline{\M}_{0,r}(\PP^n,e)$ is smooth of dimension $(n+1)(e+1)+r-4$ (see Appendix A of 
\cite{ravi} for a brief discussion of the deformation theory of $\overline{\M}_{0,r}(X,e)$).

Denote by $\mathcal C \to \overline{\M}_{0,0}(\PP^n,e)$ the 
universal curve and by $h: \mathcal C \to \PP^n$ the universal map
$$\xymatrix{ \mathcal C \ar[r]^{h} \ar[d]^{\pi} & \PP^n \\ \overline{\M}_{0,0}(\PP^n,e).}$$
For any $d \geq 1$, the line bundle $h^*\cO_{\PP^n}(d)$ is the pullback of a globally generated line bundle and so it is 
globally generated. The first cohomology group of a globally generated line bundle over a 
nodal curve of genus zero vanishes, so by the theorem of 
cohomology and base change \cite[Theorem 12.11]{hartshorne1},
$E:=\pi_*h^*(\oplus_{i=1}^m\cO_{\PP^n}(d_i))$ is a locally free sheaf of rank $de+m$. 

If  $s  \in \oplus_{i=1}^m H^0(\PP^n, \cO_{\PP^n}(d_i))$ is a section whose zero locus is $X$, then 
$\pi_*h^*s$ is a section of $E$ whose zero locus as a closed substack of $\overline{\M}_{0,0}(\PP^n,e)$ 
is $\overline{\M}_{0,0}(X,e)$ (\cite[Lemma 4.5]{hrs}). The number 
$$ \dim \overline{\M}_{0,0}(\PP^n,e) - (de+m) = e(n+1-d)+n-m-3$$
is called the {\em{expected dimension}} of $\overline{\M}_{0,0}(X,e)$. It follows that the dimension of every 
component of $\overline{\M}_{0,0}(X,e)$ is at least the expected dimension, and if the equality holds, then $\overline{\M}_{0,0}(X,e)$ is a local complete intersection substack of $\overline{\M}_{0,0}(\PP^n,e)$. 

Similarly, $\overline{\M}_{0,1}(\PP^n,e)$ is a smooth stack of dimension $(e+1)(n+1)-3$, and 
$\overline{\M}_{0,1}(X,e)$ is the zero locus of a section of a locally free sheaf of rank $de+m$ on $\overline{\M}_{0,1}(\PP^n,e)$. Therefore 
if $\dim \overline{\M}_{0,1}(X,e)= e(n+1-d)+n-m-2$, then it is a local complete intersection stack. 

The number $e(n+1-d)+n-3-m$ can be also obtained as an Euler characteristic: if 
$C$ is a smooth rational curve of degree $e$ in $X$, and if $N_{C/X}$ denotes the normal bundle of $C$ in $X$, then 
$$\chi(N_{C/X}) = \chi(T_X|_C) - \chi(T_C) = e(n+1-d)+n-m-3.$$

\subsection{Outline of proof of Theorem \ref{main}} Let $X$ be a
general hypersurface of degree $d$ in $\PP^n$, and consider 
the evaluation morphism 
$$ev: \overline{\M}_{0,1}(X,e) \to X.$$
To show that $ev$ is flat, it suffices to show that every fiber of $ev$ has dimension $e(n+1-d)-2$ (see the discussion in the beginning of Section 5).
Let $e=2$, and assume to the contrary that there is a point $p$ in $X$ and an irreducible component $\M$ of $ev^{-1}(p)$ whose dimension is 
larger 
than $2(n-d)$. 
%Since the spaces of reducible conics through $p$ in X and double covers of lines through $p$ in X 
%have dimension at most $2(n-d)-1$ 
It follows from \cite[Theorem 2.1]{hrs} that a general map parametrized by $\M$ is an 
isomorphism onto a smooth conic passing through $p$. Let 
$(\PP^1, f: \PP^1 \to X)$ be a general map parametrized by $\M$, and let $C$ denote 
the image of $f$. 
%There is a short exact sequence of normal bundles 
%$$ 
%0 \to N_{C/X} \to N_{C/\PP^n} \to \cO_C (d) \to 0,
%$$
We have 
$$\chi(N_{C/X}(-p)) = \chi (T_X|_C(-p)) - \chi(T_C(-p)) = 2(n-d).$$ 
By \cite[II, Theorem 1.7]{kollar} the Zariski tangent space to 
$ev^{-1}(p)$ at $(\PP^1, f)$ is isomorphic to $H^0(C, N_{C/X}(-p))$, 
therefore if $\dim \M > 2(n-d)$, 
then $h^0(C, N_{C/X}(-p)) > 2(n-d)$ and $H^1(C, N_{C/X}(-p)) \neq  0$.

There is a short exact sequence 
$$ 
0 \to N_{C/X}(-p) \to N_{C/\PP^n}(-p) \to \cO_C (d)(-p) \to 0,
$$
and we show that under the above assumptions, there is a subspace $W$ of codimension at most 
$n-1$ in $H^0(\PP^n, \cO_{\PP^n}(d)(-p))$ such that 
for any  $w$ in $W$ and for a general stable map $(\PP^1, f: \PP^1 \stackrel{\sim}{\rightarrow} C)$ parametrized by $\M$, 
$w|_C$ is in the image of the map 
$$\rho: H^0(C, N_{C/\PP^n}(-p)) \to H^0(C, \cO_C(d)(-p))$$ 
obtained from the above sequence (Proposition \ref{summary}). We then show in Sections 4 and 5 
that if 
$n <  \min \left( {{n-d} \choose {2}}, (d-1) \lfloor \frac{n-d}{2}
\rfloor+1\right),$  
then the existence of such $W$ implies that for a general $(\PP^1, f: \PP^1 \stackrel{\sim}{\rightarrow} C)$ parametrized by $\M$, the map 
$\rho$ is surjective. Applying the long exact sequence of cohomology to  the above short exact sequence 
we get $H^1(C, N_{C/X}(-p)) = 0$, which is a contradiction.

\section{Deformations of rational curves}

We fix a few notations for normal sheaves first. If $Y$ is a closed
subscheme of a smooth variety $Z$, as usual we write $N_{Y/Z}$ for the normal sheaf of
$Y$ in $Z$. More generally, suppose that $f: Y \to Z$ is a morphism
between  quasi-projective varieties and $Z$ is smooth.  
Denote by $T_Y$ and $T_Z$ the tangent sheaves of $Y$ and $Z$, and denote by 
$N_f$ the cokernel of the induced map $T_Y \to f^*T_Z$. We refer to $N_f$ as the 
{\em normal sheaf}  of $f$. We may sometimes write $N_{f,Z}$ instead
to emphasize the range. If $Y$ and $Z$ are both smooth and $f$ is generically
finite then the  exact sequence  
$ T_Y \to f^*T_Z \to N_{f} \to 0$
is exact on the left. If $g:Z\to T$ is another morphism to a smooth variety $T$, then we get
an exact sequence on $Y$ of normal sheaves,
\begin{equation}\label{normalsheaves}
  N_f\to N_{gf}\to f^*N_g\to 0.
\end{equation}
 
Let $B$ and $X$ be  smooth quasi-projective varieties.
Suppose that  $\pi:Y \to B$ is a smooth projective morphism and denote
by $Y_b$ the fiber over $b\in B$.  Let $F: Y \to B \times X$  be a 
morphism over $B$ such that the restriction of $F$ to every fiber of $\pi$ is 
generically finite. Let $f=F_{|Y_b}:Y_b\to X$ and $p_B$ (resp. $p_X$) be
the projections from $B\times X$ to $B$ (resp. $X$). Notice that
$p_B\circ F=\pi$ and $T_{B\times X}$ is naturally isomorphic to $p_B^*
T_B\oplus p_X^* T_X$. Thus, we have a natural map 
\begin{equation}
  \alpha:\pi^*T_B\to N_F.
\end{equation}

If $b\in B$, then we have 
$N_{F}|_{Y_b} = N_{f}$ using the following commutative diagram.
$$\xymatrix{
& 0 \ar[d] & 0 \ar[d]  \\
0 \ar[r] & T_{Y_b} \ar[d] \ar[r] & f^*T_{X} \ar[r] \ar[d] & N_{f} \ar[d]^{=} \ar[r] & 0 \\
0 \ar[r] & T_{Y}|_{Y_b} \ar[r] \ar[d] & F^*T_{B \times X}|_{Y_b} \ar[r] \ar[d] & N_{F}|_{Y_b}  \ar[r] & 0 \\
& N_{Y_b/Y} \ar[r]^{=} \ar[d] & \pi^*T_{B}|_{Y_b} \ar[d]\\
& 0 & 0 
}
$$

Thus we get, restricting $\alpha$, a map $\pi^*T_{B,b}\to N_f$
and since one has a natural map $T_{B,b}\to H^0(Y_b,\pi^*T_{B,b})$, we
get a map,
\begin{equation}
 \alpha_b: T_{B,b}  \to  H^0(Y_b, N_{f}). 
\end{equation}

%Assume now that $Z$ is a closed subscheme of $Y$ such that $\pi|_{Z}: Z \to B$ is also 
%smooth and the restriction of $F$ to every fiber of $\pi|_Z: Z \to B$ is generically finite. Let $G:Z \to B \times X$ be the restriction of $F$ to $Z$ and $g:Z_{b} \to X$ its restriction to 
%$Z_b:= \pi|_Z^{-1}(b)$. Then corresponding to the morphism $G: Z \to B \times X$, we get another natural map 
%$$
%\beta: T_{B,b} \to H^0(Z_b, N_{g}).
%$$ 
%Restricting $\alpha$ to $Z_b$, we get a map $\alpha|_{Z_b}$:
%$$
%T_{B,b}  \to H^0(Z_b, F^*T_{B \times X}|_{Z_b}) \to H^0(Z_b, N_F|_{Z_b}) = 
%H^0(Z_b, N_{f}|_{Z_b}).
%$$
%From the short exact sequences 
%$$
%0 \to T_{Z_b} \to g^*T_{X} \to N_g \to 0
%$$
%and the exact sequence 
%$$
%T_{Y_b}|_{Z_b} \to f^*T_X|_{Z_b} \to N_f|_{Z_b} \to 0,
%$$ 
%we get a map  $N_{g} \to N_{f}|_{Z_b}$.
%If $\gamma: H^0(Z_b, N_{g}) \to H^0(Z_b, N_{f}|_{Z_b})$ is the map induced on the global sections, 
%then it follows from the definitions that $\gamma \circ \beta = \alpha|_{Z_b}$.

%\medskip

\subsection{Morphisms from $\PP^1$ to general complete intersections}
Let $d_1 \leq \dots \leq d_m$ be positive integers.  For the rest of this section, we fix the following notation:

\begin{enumerate}
\item $\mathcal H$ denotes the variety  parametrizing complete intersections in $\PP^n$ which are of multidegree $(d_1, \dots, d_m)$.
\item $\mathcal U \subset \mathcal H \times \PP^n$ denotes the universal family over $\mathcal H$.
\item  $\pi_1: \mathcal U \to \mathcal H$ and $\pi_2: \cU \to \PP^n$ denote the two projection maps. 
\end{enumerate}

\begin{remark}
  If $h\in\HH$ and $\cU_h\subset\PP^n$ is the corresponding scheme, then we have
  the map $\alpha_h:T_{\HH,h}\to H^0(\cU_h, N_{\cU_h/\PP^n})$ as defined in the
  previous section. (Easy to check that this is defined even though
  $\cU\to\HH$ is not smooth). This map is an isomorphism for any $h\in \HH$.
\end{remark}

Suppose that $B$ is an smooth irreducible  quasi-projective variety and
$\psi:B\to \HH$ 
a dominant morphism. Let $\cU_B\subset B\times\PP^n$ be the fiber
product. Let 
$\pi:Y\to B$ be a dominant morphism whose fibers are smooth connected
projective curves, and let $F:Y\to \cU_B$ be a morphism over $B$ which
is generically finite on each fiber of $\pi$. 
%Let $N_{F,\cU_B}$
%(resp. $N_{F,\PP^n}$) denote the normal sheaf for the map
%$F:Y\to\cU_B$ (resp. $F:Y\to B\times\PP^n$). 
We have an exact
sequence using the sequence (\ref{normalsheaves}),
\begin{equation}\label{normalbundle}
0\to N_{F,\cU_B}\to N_{F,B\times \PP^n}\to F^*N_{\cU_B/B\times\PP^n}\to 0.  
\end{equation}
The natural map
$\alpha:\pi^*T_B\to N_{F,B\times \PP^n}$ by composing gives a map 
 \begin{equation}
   \beta:\pi^*T_B\to  F^*N_{\cU_B/B\times\PP^n} .
 \end{equation}
The dominance of $\psi$ and
the above remark show 
that for a general point $b\in B$ and $X$ the corresponding complete
intersection subscheme of $\PP^n$,
one has a surjection $T_{B,b}\to H^0(X,N_{X/\PP^n})$.

Fix a point $b\in B$ and let $C\subset Y$ be the fiber over $b$. Let
$f=F_{|C}$ and let $X$ be the complete intersection scheme
corresponding to $b$. 
Then the exact sequence (\ref{normalbundle}) specializes to
\begin{equation}\label{specialnormal}
0\to N_{f,X}\to N_{f,\PP^n}\to f^*N_{X/\PP^n}=\oplus_{i=1}^m f^*\cO_X(d_i)\to 0 
\end{equation}

\medskip
\begin{proposition}\label{liftable}
If $b$ is a general point of $B$, then the image of the pull-back 
map 

$$
H^0(\PP^n, \oplus_{i=1}^{m}\cO_{\PP^n}(d_i)) \to H^0(C, \oplus_{i=1}^mf^*\cO_{X}(d_i))
$$ 

\medskip\noindent is contained in the image of the map 
$$H^0(C, N_{f,\PP^n}) \to H^0(C, \oplus_{i=1}^mf^*\cO_{X}(d_i))$$
obtained from the above  
short exact sequence.
\end{proposition}

\begin{proof}
Since $b$ is general, it is a smooth point of $B$, so we have 
$\alpha_b: T_{B,b} \to H^0(C, N_{f,\PP^n}) $ as before.
We have a commutative diagram 
$$
\xymatrix{
T_{B,b} \ar[r]^{\alpha_b} \ar@{>>}[d]^{d  \psi} & H^0(C, N_{f,\PP^n}) \ar[d] \\
T_{\HH,[X]} = H^0(X, \oplus_{i=1}^m\cO_{X}(d_i)) \ar[r] & H^0(C, \oplus_{i=1}^mf^*O_{X}(d_i))}
$$

\noindent and $d \psi$ is surjective since $\psi$ is dominant and $b$ is a general point of $B$, so the result follows.
\end{proof}

\medskip

A curve $C \subset \PP^n$ is called {\em $d$-normal} if the restriction map 
$$
H^0(\PP^n, \cO_{\PP^n}(d')) \to H^0(C, \cO_{C}(d'))
$$ 
is surjective for every $d' \geq d$. 

\begin{corollary}\label{normal}
Suppose that $X \subset \PP^n$ is general complete intersection of multidegree $(d_1,\dots,d_m)$, 
$d_1 \leq \dots \leq d_m$. If $C$ is any $d_1$-normal smooth rational curve 
of degree $e$ in $X$, then $H^1(C,N_{C/X})=0$. 
In particular, $R_e(X)$ is smooth at $[C]$.
\end{corollary}

\begin{proof}
If $f: \PP^1 \to C \subset X$ is an isomorphism, then $H^1(\PP^1, N_{f,\PP^n})=0$. 
So applying  the long exact sequence of cohomology 
to the sequence of normal sheaves
$$ 
0 \to N_{f,X} \to N_{f,\PP^n} \to \oplus_{i=1}^{m}f^*\cO_{X}(d_i) \to 0,
$$ 
the statement follows from Proposition \ref{liftable}. 
\end{proof}

\medskip

Now, we make a further assumption that there exists a morphism
$\phi:B\to\cU$ so that the composition with the map $\pi_1: \cU\to\HH$ is
$\psi$. This is equivalent to saying that we are given a section
$\sigma$ for the map $\cU_B\to B$.

Next, for a point $([X],p)$ in the image of $\phi: B \to \mathcal U$, we define a subspace $W_{X,p}$ of 
$H^0(X, N_{X/\PP^n})$ as follows.  
For $p \in \PP^n$, let $B_p = (\pi_2\circ \phi)^{-1}(p)$ with the reduced induced structure, and let 
$\HH_p \subset \HH$ be the 
closure of $\psi(B_p)$.
If $([X],p)$ is in the image of $\phi$ , then $[X] \in \HH_p$. We can identify 
$T_{\HH, [X]}$ with $H^0(X, N_{X/\PP^n})$ and we define $W_{X,p}
\subset H^0(X, N_{X/\PP^n} )$ to 
be $T_{\HH_p,[X]}$ under this identification.

Assume now that $b$ is a general point of $B$ and $\phi(b) = ([X],p)$. Since every complete intersection which is 
parametrized by $\HH_p$ contains $p$, 
we have $$W_{X,p} \subset H^0(X, N_{X/\PP^n}(-p)).$$ 
Since $\psi$ is dominant by our assumption, the codimension of $\HH_p$ in $\HH$ is at most $n$. Thus  $W_{X,p}$ is of codimension $\leq n$ 
in $H^0(X, N_{X/\PP^n})$, and it is of codimension $\leq n-m$ 
in $H^0(X, N_{X/\PP^n}(-p))$. 

We make a further assumption on $F:Y\to \cU_B$ that $\sigma(B)\subset
F(Y)$. 
Let $C$ be the fiber of $\pi$ over $b$ and $f:C \to X$ the restriction of $F$ to $C$, so the image of $f$ is a curve which passes through $p$. Let  
$D = (f^{-1}(p))_{\rm{red}} \subset C$. 
% Since $W_{X,p} \subset H^0(X, \oplus_{i=1}^m\cO_X(d_i)(-p))$,  
Then the image of the pullback map 
$$
W_{X,p} \to H^0(C, f^*N_{X/\PP^n})
$$ 
is contained in $H^0(C, f^*N_{X/\PP^n}(-D))$. 
Consider the 
short exact sequence of normal sheaves (\ref{specialnormal}) twisted
with $\cO_{C}(-D)$:   
$$
0 \to N_{f,X}(-D) \to N_{f,P^n}(-D) \to f^*N_{X/\PP^n}(-D) \to 0.
$$

\noindent We claim that the image of the pull-back  map 
$$
W_{X,p}  \to H^0(C, f^*N_{X/\PP^n}(-D))
$$ 
is contained in the image of the map 
$$H^0(C, N_{f, \PP^n}(-D)) \to H^0(C,f^*N_{X/\PP^n}(-D))$$ 
obtained from the above short exact sequence. Since $b$ is a general point,
and since the morphism $B_p \to \HH_p$ is dominant, the induced map on the 
Zariski tangent spaces $T_{B_p, b} \to T_{\HH_p, [X]}=W_{X,p}$ is surjective. 
Set 
$$
Y_p = \pi^{-1}(B_p),
$$ 
and let $\alpha_b: T_{B_p, b} \to H^0(C, N_{f,\PP^n})$ be the 
natural map corresponding to the morphism $Y_p \to B_p \times \PP^n.$
%Consider the restriction of $F$ to $Y_p$: 
%$$F_p: Y_p \to B_p \times \PP^n,$$
%and let $\alpha: T_{B_p, b} \to H^0(C, N_{f,\PP^n})$ be the 
%natural map corresponding to $F_p$. 
Set 
$$
Z_p=  (F^{-1}(B_p \times \{p\}))_{\rm{red}} \subset Y_p.
$$
Then the fiber of $Z_p$ over $b$ is $D$. Let $g: D \to X$ be the restriction of $f$ to $D$, and let 
$\beta_b: T_{B_p,b} \to H^0(D, N_{g, \PP^n}) = H^0(D, f^*T_X|_D)$ be the 
natural map corresponding to the morphism 
$ Z_p \to B_p \times \PP^n$.  
Then $\beta_b$ is clearly the zero map. The claim now follows from the following commutative diagram
$$
\xymatrix{
T_{B_p, b} \ar@{>>}[rr]^{} \ar[dd]_0 \ar[rd]^{\alpha}  & & W_{X,p} \subset H^0(X, 
N_{X/\PP^n}) \ar[d]\\
&H^0(C, N_{f,\PP^n}) \ar[d] \ar[r] & H^0(C, f^*N_{X/\PP^n})\\
f^*T_{\PP^n}|_D \ar[r] & N_{f,\PP^n}|_{D}.
}
$$
\bigskip

We summarize the above discussion in the following proposition.

\begin{proposition}\label{summary}
For a general point $([X],p)$ in the image of $\phi$, there is a subspace $$W_{X,p} \subset 
H^0(X, N_{X/\PP^n}(-p))$$ of codimension at most $n-m$ 
and an open subset $B_0 \subset \phi^{-1}([X],p)$ with the following property: 
for every $b\in B_0$, if we denote the fiber of $\pi$ over $b$ by $C$, the 
restriction of $F$ to $C$ by 
$f:C \to X$, and the inverse image of $f^{-1}(p)$ with the reduced induced structure by $D$, then 
for every $w \in W_{X,p}$, $f^*w$ can be lifted to a section of $H^0(C,N_{f,\PP^n}(-D))$. 
\end{proposition}

\bigskip

%%%%%%%%%%%%%%%%%%%%%%%%%%%%%%%%%%%%%%%%%%%%
\section{Conics in projective space}
%%%%%%%%%%%%%%%%%%%%%%%%%%%%%%%%%%%%%%%%%%%%

Let $\Hilb_{2t+1}(\PP^n)$ denote the Hilbert scheme of subschemes of $\PP^n$ with 
Hilbert polynomial $2t+1$. In this section we prove the following:

\begin{proposition}\label{zero}
Suppose that $p$ is a point of $\PP^n$ and $R \subset \Hilb_{2t+1}(\PP^n)$ is an irreducible projective subscheme of dimension $r$ such that every 
curve parametrized by $R$ is a smooth conic through $p$.  
Let $W\subset H^0(\PP^n, \cO_{\PP^n}(d)(-p))$ be a subspace of codimension $c$. 
If 
$$c < \min \left ( {{r+1} \choose {2}}, (d-1) \left\lfloor
\frac{r+1}{2} \right\rfloor+1 \right ),$$ 
then for a general $[C] \in R$,
and for every $2 \leq k \leq 2d$, the image of the restriction map  
$$ W \to H^0(C, \cO_C(d))$$
contains a section of $\cO_C(d)$ which has a zero of order $k$ at $p$.  
\end{proposition}

We start with  a lemma.

\begin{lemma}\label{osculating}
If $l$ is a line in $\PP^n$ through $p$, then there is no 
complete one-dimensional family of smooth conics through $p$ all tangent to $l$.
\end{lemma}

\begin{proof}
Assume to the contrary that there is such a family $B$. By passing to a desingularization we can assume that $B$ is smooth. 
Let $Y \subset B \times \PP^n$ be the universal family over $B$ and $g: Y \to \PP^n$ the projection  map. 

The point $p$ gives a section $\sigma_p$ of the family $Y \to B$. 
Fix a point $q \neq p$ on $l$. Then for every conic $C$ in the family, there is a unique line $l' \neq l$ such that $l'$ passes through $q$ and is tangent to 
$C$. Let $\sigma_q$ be the section of $Y \to B$ such that $\sigma_q([C])$ is the point of intersection of $C$ and $l'$. 
Since if $q_1$ and $q_2$ are two distinct points of $l$, $\sigma_{q_1}$ and $\sigma_{q_2}$ are disjoint, $Y \simeq B \times \PP^1$. 
Since the section $\sigma_p$ is contracted 
by $g$, by the rigidity lemma \cite{mumford}, 
$g$ factors through the projection map $B \times \PP^1 \to \PP^1$. Thus the image of $g$ should be 
one-dimensional which is a contradiction.   

\begin{remark}
The following argument shows 
that the set-theoretic map $\sigma_q$ defined in the proof of the above
lemma is a section. If $q\neq p$ on $l$,
$B\times\{q\}\cap Y=\emptyset$. Thus the projection from $q$ defines a
morphism $g:Y\to B\times\PP^{n-1}$. For any point $b\in B$, this map is
just a map from a conic to a line and thus $g$ is a two-to-one map to
its image. Let $R\subset Y$ be the ramification locus. Then, the map
$R\to B$ is a double cover. But, $B\times\{p\}\subset R$ and the 
residual part is a section of $Y\to B$ and just $\sigma_q$.
\end{remark}
\end{proof}

\begin{corollary}\label{non-degenerate} Suppose that $R \subset \Hilb_{2t+1}(\PP^n)$ is a closed subscheme such that every 
curve parametrized by $R$ is a smooth conic through $p$. Then,

\begin{enumerate}
\item $\dim R \leq n-1$.
\item If the $2$-planes spanned by the curves parametrized by $R$ all pass through a point $q \neq p$, then $\dim R \leq 1$.
\end{enumerate}
\end{corollary}

\begin{proof}
(a) If $\dim R\geq n$, since the lines through $p$ form a $\PP^{n-1}$, if
we associate for any $r\in R$ the tangent line through $p$ to the
conic corresponding to $r$, then by dimension considerations, there is
a positive dimensional closed subfamily of $R$ with the same tangent
line, which is impossible by the previous lemma. 

(b) Set $ r= \dim  R$. Let $l$ be the line through $p$ and $q$, and let $R'$ be the closed subscheme of $R$ parametrizing 
conics tangent to $l$. By Lemma \ref{osculating} if $R'$ is not empty, it is zero dimensional. Since every conic parametrized by the complement of 
$R'$ intersects $l$ in a point other than $p$, there should be a point $q'\in l$, and a closed subvariety of dimension $r-1$ in $R$ 
parametrizing conics passing through $p$ and $q'$. By \cite[Lemma 5.1]{hrs}, in any projective 
1-dimensional family of conics passing through $p$ and $q$, there should be reducible conics, 
so $r \leq 1$.

\end{proof}
\medskip

With the hypothesis in the Proposition \ref{zero}, we fix the
following notation.
Let $\Gamma$ be a hyperplane in $\PP^n$ which does not pass through $p$. 
Choose a homogeneous system of coordinates for $\PP^n$ so that 
$p = (1:0: \dots :0)$ and $\Gamma = \{x_0=0\}$.  For $[C] \in R$, denote by 
$q_C$ the intersection of $\Gamma$ with 
the line tangent to $C$ at $p$, and denote by $l_C$ the line of intersection of $\Gamma$ with the $2$-plane  
spanned by $C$.  Notice that $q_C \in l_C$. Denote by $\Sigma$ the cones of lines tangent at $p$ to the conics parametrized by $R$ and denote by $Y$ the intersection of 
$\Sigma$ and $\Gamma$. By Lemma \ref{osculating}, $\dim Y = r$.

For $1 \leq i \leq d$, multiplication by $x_0^{d-i}$ identifies $H^0(\Gamma, \cO_{\Gamma}(i))$ 
with a subspace of $H^0(\PP^n, \cO_{\PP^n}(d)(-p))$. Let $W_i \subset H^0(\Gamma, \cO_{\Gamma}(i))$ be the intersection of $W$ 
with $H^0(\Gamma, \cO_{\Gamma}(i))$ under this identification. Since the codimension of 
$W$ in $H^0(\PP^n, \cO_{\PP^n}(d)(-p))$ is $c$, the codimension of 
$W_i$ in $H^0(\Gamma, \cO_{\Gamma}(i))$ is $\leq c$. 

\begin{lemma}\label{order-of-zero}
If $f \in H^0(\Gamma, \cO_{\Gamma}(i))$ is such that $f|_{l_C}$ has a zero of order 
$j$ at $q_C$, then the restriction of $x_0^{d-i}f$, considered as a section of 
$H^0(\PP^n, \cO_{\PP^n}(d))$, to $C$  has a zero of order 
$i+j$ at $p$.
\end{lemma}
\begin{proof}
Let $P$ be the two plane spanned by $C$. Then the divisor of $x_0$ in
this plane is just $l_C$. The divisor of $f|_{l_C}$ is $jq_C+E$ where
$E$ is an effective divisor of degree $i-j$ whose support does not
contain $q_C$. Then the divisor in $P$
of $x_0^{d-i}f$ is $(d-i)l_C+jT+E'$ where $T$ is the tangent line of
$C$ at
$p$, $E'$ is a union of $(i-j)$ lines passing through $p$, none of
them equal to $T$. Thus, the
order of its
restriction to $C$ at $p$ is just $2j+(i-j)=i+j$.
%If we consider the restriction of $x_0^{d-i}f$ to the $2$-plane $P$
%spanned by $C$, then its zero locus $D$  
%is the union of the lines through $p$ and the zero locus of
%$f|_{l_C}$. The line through $p$  
%and $q_C$ intersects $C$ at $p$ with multiplicity 2, and the rest of
%the lines intersect $C$ 
%at $p$ with multiplicity 1, so the multiplicity of the intersection of
%$D$ and $C$ at $p$ is  
%$2j + (i-j) = j+i$.
\end{proof}

\begin{lemma}\label{codimension}
Suppose that $Z$ is a $k$-dimensional irreducible subvariety of $\PP^n$, and let $I_Z$ be the ideal sheaf of $Z$ in $\PP^n$. 
\begin{itemize}
\item[(i)] The codimension of 
$H^0(\PP^n, I_{Z}(t))$ in $H^0(\PP^n, \cO_{\PP^n}(t))$ 
 is at least ${{t+k}\choose{k}}$. 
\item[(ii)] If $k\geq 1$, and if $Z$ spans a linear subvariety of
  dimension $s$ in $\PP^n$, then  
the codimension of $H^0(\PP^n, I_{Z}(t))$ in $H^0(\PP^n, \cO_{\PP^n}(t))$ is at least $st+1$.
\end{itemize}
\end{lemma}

\begin{proof}
(i) If $\pi: Z \to \PP^k$ is a general linear projection, then $\pi$ is a finite map, so 
the induced map $\pi^*: H^0(\PP^k, \cO_{\PP^k}(t)) \to H^0(Z, \cO_Z(t))$ is injective. 
We have a commutative diagram 
$$
\xymatrix{
H^0(\PP^k, \cO_{\PP^k}(t)) \ar@{^{(}->}[r] \ar@{^{(}->}[d]^{\pi^*} & H^0(\PP^n, \cO_{\PP^n}(t)) \ar[dl]\\
H^0(Z, \cO_Z(t))
}
$$

\noindent so  the codimension of the kernel of the restriction map $H^0(\PP^n, \cO_{\PP^n}(t)) \to
H^0(Z, \cO_Z(t))$ is $\geq \dim H^0(\PP^k, \cO_{\PP^k}(t)) = {{t+k}\choose{k}}$. 
%Assume $\Gamma$ is given by $x_0=x_1=\dots=x_k=0$. 
%For any non-zero form $f$ of degree $t$ in $x_0, \dots, x_k$, choose a 
%linear subvariety $\Gamma'$ of codimension $k$ in $\PP^n$ such that $\Gamma \subset \Gamma'$ and 
%$f$ does not vanish on $\Gamma'$. If $q \in \Gamma' \cap Z$, then 
%$f$ does not vanish at $q$. So there is a subspace of $H^0(\PP^n, \cO_{\PP^n}(t))$ of dimension ${{t+k}\choose{k}}$ 
%such that none of its nonzero elements vanish on $Z$.

(ii) Let $C\subset Z$ be an irreducible curve whose span is equal to
the span of $Z$, which can always be achieved by taking an irreducible
curve 
passing through finitely many linearly independent points on $Z$. Since
$I_Z\subset I_C$, it is clear that we need to prove the lemma only for
$C$ and thus we may assume $k=1$. 

We can assume that the codimension 2 subvariety of $\PP^n$ defined by $\{x_0=x_1=0\}$ 
does not 
intersect $Z$. Then the surjective map 
$$
\xymatrix{
\cO_Z ^{\oplus t}   \ar[rrr]^{(x_0^{t-1}, x_0^{t-2}x_1,\dots, x_1^{t-1})} &&& \;\;\;\; \cO_Z(t-1)
}
$$ 
gives a short exact sequence 
$$
0 \to \cO_{Z}^{\oplus t-1} \to \cO_Z(1)^{\oplus t} \to \cO_Z(t) \to 0.
$$
Since $Z$ spans a linear subvariety of dimension $s$, 
the image of the restriction map $H^0(\PP^n, \cO_{\PP^n}(1)^{\oplus t}) \to H^0(Z, \cO_Z(1)^{\oplus t})$ 
has dimension at least $(s+1)t$, so the image of the map 
$H^0(\PP^n, \cO_{\PP^n}(1)^{\oplus t}) \to H^0(Z, \cO_Z(t))$ has dimension at least $(s+1)t-(t-1) 
=st+1$. Therefore the image of the restriction map $H^0(\PP^n, \cO_{\PP^n}(t)) \to H^0(Z, \cO_Z(t))$ 
is at least $(st+1)$-dimensional.  
\end{proof}

\begin{proof}[Proof of Proposition \ref{zero}] 
Since $\dim Y = r$, by part (i) of the above lemma, the codimension of the space of sections in 
$H^0(\Gamma, \cO_{\Gamma}(i))$ which vanish on $Y$ is at least ${{r+2}\choose{2}}$ 
for $2 \leq i \leq d$. Since the codimension of $W_i$ in $H^0(\Gamma, \cO_{\Gamma}(i))$ is 
smaller than ${{r+2}\choose{2}}$, for every $2 \leq i \leq d$, 
there is $g_i \in W_i $ which does not vanish on $Y$, so for a general $C$ parametrized by $R$, 
$g_i$ does not vanish at $q_C$. 
Hence $x_0^{d-i}g_{i} \in W$, and its restriction to $C$ has a zero of
order $i$ at $p$ by  lemma \ref{order-of-zero}. 

To complete the proof of Proposition \ref{zero}, we show that for every $1 \leq j \leq d$, and for a general conic $C$ parametrized by $R$, 
there is $f_j \in W_d $ such that $f_j|_{l_C}$ 
has a zero of order $j$ at $q_C$, so $f_j|_C$ has a zero of order $d+j$ at $p$ by Lemma \ref{order-of-zero}. 
This is proved in the next proposition. 

\end{proof}

\begin{proposition}\label{conic}
Let $p$ be a point of $\PP^n$ and $\Gamma$ a hyperplane in $\PP^n$ which does not pass 
through $p$. Suppose that $R \subset \Hilb_{2t+1}(\PP^n)$ is an irreducible closed subscheme of dimension $r$ such that every curve parametrized by $R$ is smooth 
and passes through $p$. Let $W_d $ be a subspace of  $H^0(\Gamma, \cO_{\Gamma}(d))$, $d \geq 2$,  of codimension $c'$. If
$$ 
c' < \min \left ( {{r+1} \choose {2}}, (d-1) \left\lfloor
\frac{r+1}{2} \right\rfloor+1 \right ), 
$$ then for a general $[C]$ parametrized by $R$, and for every $0 \leq k \leq d$, 
there is $f_k \in W_d$ 
such that $f_k|_{l_C}$ has a zero of order equal to $k$ at $q_C$.
 \end{proposition}

\begin{proof}
Denote by $Y \subset \Gamma$  the subvariety swept out by the points
$q_C$ for $[C] \in R$ as before. 
By Lemma \ref{osculating}, $\dim Y = r$. 

Assume first that $0 \leq k \leq d-2$. Let $H$ be a general hyperplane in $\Gamma$, and let $Y' = Y \cap H$. Let $R^{\prime}$ 
be the locus in $R$ parametrizing conics $C$ for which $q_C \in Y'$. Then Lemma 
\ref{osculating} shows that $\dim R' = \dim Y' =r-1$. Choose a system
of homogeneous coordinates for $\PP^n$ as before so that $p =
(1:0:\dots:0)$, $\Gamma$ is given by $x_0 = 0$, and $H$ is given by  
$x_0=x_1=0$. Consider the vector space $U$ of all polynomials of the form $x_1^{k}f$ where $f$ is a homogeneous polynomial of degree $d-k \geq 2$ in $x_2, \dots, x_n$, and let $U_0$ be the subspace of $U$ 
consisting of those  
$x_1^kf$ where $f$ vanishes on $Y'$. Then the codimension of $U_0$ in $U$ is $\geq {{(r-1)+2} \choose {2}}$ by Lemma \ref{codimension} (i). 
Therefore if ${{r+1} \choose {2}}$ is greater than the codimension of $W_d$ in $H^0(\Gamma, \cO_{\Gamma}(d))$, then there is an element of the form $x_1^kf$ in $W_d$ 
such that $f$ does not vanish on $Y'$. So for a general $C$ parametrized by $R'$, $f$ does not 
vanish on $q_C$. Also, $x_1$ does not vanish on $l_C$ since by the next lemma  
$l_C$ does not lie on $H$. Therefore, the restriction of $x_1^kf$ to $l_C$ has a zero of order $k$ at $q_C$ for a general $C$ parametrized by 
$R'$. 

\begin{lemma}\label{linear}
Suppose that 
$R$ is an $r$-dimensional family of  smooth conics through $p$. Then for a general codimension $m$ 
linear subvariety $\Lambda$ of $\Gamma$, the locus in $R$ parametrizing conics $C$ for which $l_C$ lies on $\Lambda$ 
has dimension $\leq r-2m$. 
\end{lemma}

\begin{proof}
Denote by $G := \Gr(n-1-m,n-1)$ the Grassmannian of linear subvarieties of codimension $m$ in $\Gamma$. 
Consider the incidence correspondence 

$$
\xymatrix{
&  \;\;\;\;\; I = \{ ([C], \Lambda), l_C \subset \Lambda \} \subset R \times G \ar[ld]^{\pi_1}  \ar[rd]^{\pi_2}  &
\\
R  & & G}
$$

\noindent where 
$\pi_1$ and $\pi_2$ are the projection maps. The fibers of $\pi_1$ are of dimension $m(n-m-2)$ and $\dim G = m(n-m)$. So for a general $[\Lambda] \in G$, 
$\dim \pi_2^{-1}([\Lambda]) \leq \dim I - \dim G = \dim R+m(n-m-2)-m(n-m) = r-2m$.

\end{proof}

\smallskip

Next we prove the statement for $k = d-1$ or $d$. We need the following lemma.

\begin{lemma}\label{independent} 
Suppose that $R \subset \Hilb_{2t+1}(\PP^n)$ is a closed subscheme such that every 
curve parametrized by $R$ is a smooth conic through $p$. Let $\dim R = r$, and let $s = \lfloor \frac{r+1}{2} \rfloor$. 
If $C_1, \dots, C_s$ are general conics parametrized by $R$, then $l_{C_1}, \dots, l_{C_s}$ 
are linearly independent,  i.e. they span a  linear subvariety of dimension $2s-1$.
\end{lemma}
\begin{proof}
Assume that $s'$ is the largest number for which there are conics $C_1, \dots, C_{s'}$ 
parametrized by $R$ such that $l_{C_1},\dots, l_{C_{s'}}$ are linearly independent. Let $\Sigma$ be the linear 
span of $l_{C_1}, \dots, l_{C_{s'}}$. Then for any curve 
$C$ parametrized by $R$, $l_C$ should intersect $\Sigma$.
By Corollary \ref{non-degenerate}, for every point $q$ in $\Sigma$, there is at most a 
1-dimensional subscheme of $R$ parametrizing conics $C$ such that $l_C$ passes through $q$. 
Therefore, $\dim R \leq \dim \Sigma+1 = 2s'$, and so $s=\lfloor \frac{r+1}{2} \rfloor \leq s'$.  Hence 
there are $s$ conics parametrized by $R$ whose corresponding $l_C$'s are linearly independent, and 
so the same is true for $s$ general conics parametrized by $R$.
\end{proof}

\medskip

If $r=1$, then $c'=0$ by our assumption and there is nothing to prove. So assume $r \geq 2$, and put $s = \lfloor \frac{r+1}{2} \rfloor$. Let $[C_1], \dots, [C_s]$ be general points of 
$R$. By the above lemma, $l_{C_1}, \dots, l_{C_s}$ are linearly independent.
Choose points $q'_{C_i} \neq q_{C_i}$  in $l_{C_i}$.
Denote by $H'$ the $(s-1)$-dimensional linear subvariety spanned by 
the points $q'_{C_i}$, $1 \leq i \leq s$, and let $H$ be a general linear subvariety of $\Gamma$ of codimension $s$ containing 
$q_{C_1}, \dots, q_{C_s}$ (note that $n-1-s \geq s-1$ by Corollary \ref{non-degenerate}, so such $H$ exists). Then  $H$ and $H'$ are disjoint. 
Since $C_1, \dots, C_s$ and $H$ are general and $r-s \geq 1$, by Bertini theorem the locus $R'$ in $R$ parametrizing curves $C$ such that $q_C \in H$ is irreducible. 
%(Here I am using the following version of Bertini:

%\begin{theorem}
%Suppose that $R$ is an irreducible variety of dimension $r$ and $f: R \to \PP^{n-1}$ is a finite morphism. If $p_1, \dots, p_m$ are general points of 
%$R$, and if $H$ is a general linear subvariety of codimension $s$ in $\PP^{n-1}$ containing $p_1, \dots, p_m$, then 
%$f^{-1}(H)$ is irreducible if $s < \dim f(X) = r$.   
%\end{theorem}
%\noindent I am also using the fact that there is a morphism $q: R_{\text{red}} \to \PP^{n-1}$ sending $[C]$ to $q_C$ )

For $[C]\in R'$, 
let $\Sigma_C$ be the linear subvariety spanned by 
$H$ and $l_{C}$. Since for each $1 \leq i \leq s$, $l_{C_i}$ does not lie on $H$, for a general $[C] \in R'$, $l_C$ does not lie on $H$, and so 
$\Sigma_C$ is of codimension $s-1$ in $\Gamma$ and intersects $H'$ at a point $q'_C$. Since $R'$ is irreducible, the points $q'_C$ span an irreducible 
quasi-projective subvariety $Z$ of $H'$. Since $Z$ is irreducible and contains $q'_{C_1}, \dots, q'_{C_s}$, it is non-degenerate in $H'$ and has dimension at least 1. 

Set $U := H^0(H', \cO_{H'}(d))$ which can be considered  as a subspace of
$H^0(\Gamma, \cO_{\Gamma}(d))$. 
 By Lemma \ref{codimension} (ii),  
forms of degree $d$ on $H'$ which vanish on $Z$ form a subspace of codimension at least 
$d(s-1)+1$ in $U$. Therefore if $d(s-1)+1>c'$, there is a form $f \in W_d \cap U$ which does not vanish at the generic point of 
$Z$. If $[C]$ is such that $f$ does not vanish at $q'_C$, then 
% since $f|_{\Sigma_C}$ vanishes on the divisor $H$ with multiplicity $d$, 
it does not vanish at any point of $\Sigma_C$ which is not on $H$, so $f$ cannot be identically zero on $l_{C}$.  
Hence $f \in W_d$ and $f|_{l_C}$ has a zero of order $d$ at $q_C$. 

Repeating the same argument with $d$ replaced by 
$d-1$ and choosing a form $h$ of degree $1$ on $\Gamma$ which does not vanish on $H$, 
we see that if $(d-1)(s-1)+1 > c'$, then there is a form $g$ of degree $d-1$ in $H^0(H', \cO_{H'}(d-1))$ such that 
$gh \in W_d$ and $gh|_{l_C}$ has a zero of order $d-1$ at $q_C$. This completes the proof. 
\end{proof}

\bigskip

%%%%%%%%%%%%%%%%%%%%%%%%%%%%%%%%%%%%%%%%%%%%%
\section{Proof of Theorem \ref{main}}
%%%%%%%%%%%%%%%%%%%%%%%%%%%%%%%%%%%%%%%%%%%%%

Let $X$ be a general hypersurface of degree $d$ in $\PP^n$. In this section we show that the evaluation map 
$$
ev: \overline{\M}_{0,1}(X, 2) \to X
$$
is flat of constant fiber dimension $2(n-d)$ if $d$ is in the range of Theorem \ref{main}. 
Recall that  $\overline{\M}_{0,1}(\PP^n,  2)$ is a smooth stack of dimension $3(n+1)-3=3n$ and 
that  $\overline{\M}_{0,1}(X,2)$ is the zero locus of a section of 
a locally free sheaf of rank $2d+1$ over $\overline{\M}_{0,1}(\PP^n,2)$. 
If the fibers of $ev$ are of expected dimension $2(n-d)$, then $\overline{\M}_{0,1}(X,2)$ has 
dimension 
$$
2(n-d)+n-1 = 3n-(2d+1),
$$ 
so it is a local complete intersection and in particular a Cohen-Macaulay substack of 
$\overline{\M}_{0,1}(\PP^n,2)$. Since a 
map from a Cohen-Macaulay scheme to a smooth scheme is flat if and only if it has constant fiber dimension (\cite[Theorem 23.1]{matsumura}), to prove the theorem, it is enough to show that 
$ev$ has constant fiber dimension $2(n-d)$. Note that $\dim \overline{\M}_{0,1}(X,2)$ 
is at least $3n-(2d+1)$, and $ev$ is surjective, so every irreducible component of any 
fiber of $ev$ has dimension at least $2(n-d)$.

Let $p$ be a point in $X$ and $\M$ an irreducible component of $ev^{-1}(p)$. 
%It follows from the case $d=1$ of \cite[Lemma 6.4]{hrs} 
Since $\overline{\M}_{0,1}(\PP^n,2)$ is smooth of dimension $3n$ (see Section 2), and since the fibers of the evaluation map 
$$\widetilde{ev}: \overline{\M}_{0,1}(\PP^n,2) \to \PP^n$$ 
are all isomorphic, $\widetilde{ev}^{-1}(p)$ is smooth of dimension $2n$. Since the space 
of lines through every point of $\PP^n$ has dimension $n-1$, the space of 
reducible conics through every point of $\PP^n$ has dimension $2n-1$.  
Thus the stable maps with reducible domains which are parametrized by $\widetilde{ev}^{-1}(p)$ form a divisor in $\widetilde{ev}^{-1}(p)$. 
Hence there are 3 possibilities for the locus $\M'$ in $\M$ parametrizing stable maps with reducible 
domains: 1) $\dim \M' = \dim \M-1$, 2) $\dim \M' =\dim \M$, 3) $\M'$  is empty. 
Since the space of lines through every point $p$ in $X$ has dimension $n-d-1$ (\cite[Theorem 2.1]{hrs}), the space of reducible conics through $p$ has dimension 
$2n-2d-1$, and hence in the first case, $\dim \M = 2(n-d)$ which is the expected dimension. The second case cannot happen 
since in this case $\dim \M$  would be equal to $2(n-d)-1$ which is smaller than expected. 

So assume that $\M$ parametrizes 
only stable maps with irreducible domains, and assume to the contrary 
that $\dim \M \geq 2(n-d)+1$. If $1 \leq d < \frac{n+1}{2}$, then by \cite[Theorem 1.2 and Corollary 5.5]{hrs}, the space of conics through every point of $X$ has dimension $2(n-d)$, hence 
$\frac{n+1}{2} \leq d$.

Any map parametrized by $\M$ is either an isomorphism onto a smooth conic through $p$ or a double cover of a line through $p$. 
For any line $l \subset X$ through $p$, there is a $2$-parameter family of degree 2 covers of $l$ (determined by the 2 branch points), and by 
\cite[Theorem 2.1]{hrs}, the family of lines through $p$ in $X$ has dimension $n-d-1$. So the substack of $\M$ parametrizing 
double covers of lines has dimension at most $n-d+1$, and therefore,   
there is an irreducible  closed subscheme $R$ of $\M$ of dimension $n-d-1$ which parametrizes only  
smooth embedded conics through $p$ in $X$.

Let $\HH$ be the projective space of hypersurfaces of degree $d$ in $\PP^n$, and let $B$ be the closed subscheme of 
$\HH \times \PP^n \times \Hilb(\PP^n)$ parametrizing all the points $([X], p, [C])$ such that 
$C$ is a smooth conic in $X$ through $p$ which belongs to a larger
than expected component $R$ of $ev^{-1}(p)$.
Then since by our assumption the projection map $B \to \HH$ is dominant and $X$ is general,  
by Proposition \ref{summary}, there is a subspace $W_{X,p}  \subset
H^0(X, \cO_X(d)(-p))$ of codimension at most $n-1$ such that for a general curve $[C]$ parametrized by $R$ and for every $w \in W_{X,p}$, $w|_C$ can be lifted to a section 
of $N_{C/\PP^n}(-p)$  under the map 
$$\rho: H^0(C,N_{C/\PP^n}(-p)) \to H^0(C,\cO_C(d)(-p))$$ obtained 
from the short exact sequence 
\begin{equation}\label{seq2}
0 \to N_{C/X}(-p) \to N_{C/\PP^n}(-p) \to \cO_C(d)(-p) \to 0.
\end{equation}
We show that this implies $\rho$ is surjective. Since $\frac{n+1}{2} \leq d$,  we have 
$$
{{n-d} \choose {2}} \leq (d-1)\left\lfloor \frac{n-d}{2} \right\rfloor +1.
$$ Hence by Proposition \ref{zero}, if 
$n \leq {{n-d} \choose {2}}$, then for every $2 \leq i \leq 2d$, there exists  
$s_i \in H^0(C, \cO_C(d))$ which has a zero of order $i$ at $p$ and can be lifted to a section of 
$N_{C/\PP^n} (-p)$. So to show $\rho$ is surjective, it is enough to 
show that there is also a section of $H^0(C, \cO_C(d))$ which has a zero of order 1 at $p$ and can be lifted to a section of $N_{C/\PP^n}(-p)$.
%, we can conclude that the map  
%$$
%H^0(C,N_{C/\PP^n}(-p)) \to H^0(C,\cO_C(d)(-p))
%$$ 
%is surjective. 

Suppose that $X$ is given by $f=0$ in $\PP^n$, and consider the commutative diagram
$$
\xymatrix{ \cO_C(1)^{n+1}\; \ar@{>>}[rr]^(0.55){(\frac{\partial f}{\partial x_0},\dots, \frac{\partial f}{\partial x_n})} \ar@{>>}[d] && \;\;\; \cO_C(d) \\
T_{\PP^n}|_C \ar@{>>}[rr] && N_{C/\PP^n} \ar@{>>}[u]. 
}
$$
Since $X$ is smooth at $p$, there is an $i$, $0 \leq i \leq n$, such that 
$\frac{\partial f}{\partial x_i}$ does not vanish at $p$. Hence there is a section of $\cO_C(d)$ which has a simple zero at $p$ and is in the image of 
the map $H^0(C, \cO_C(1)^{n+1}) \to H^0(C, \cO_C(d))$. 
Such a section can be lifted to a section of $N_{C/\PP^n} (-p)$. Hence $\rho$  
is surjective.

Applying the long exact sequence of cohomology to sequence
(\ref{seq2}), we get 
$
H^1(C, N_{C/X}(-p) ) = 0,
$ 
thus  
$$
h^0(C, N_{C/X}(-p))= \chi(N_{C/X}(-p)) = \chi(T_X|_C(-p)) - \chi (T_C(-p)) = 2(n-d).
$$
On the other hand, the Zariski tangent space to $ev^{-1}(p)$ at $[C]$ is isomorphic to 
$H^0(C, N_{C/X}(-p))$, thus 
$\dim \M$ should be at most $2(n-d)$, which is a contradiction.

%%%%%%%%%%%%%%%%%%%%%%%%%%%%%%%%%%%%%%%%%%
\section{Non-free lines in complete intersections}
%%%%%%%%%%%%%%%%%%%%%%%%%%%%%%%%%%%%%%%%%%

By a {\em{non-free line}} in a subvariety $X$ of $\PP^n$ we mean a line $l$ which is contained in the smooth locus of $X$ and its normal bundle in $X$ 
is not globally generated. For a line $l$ contained in the smooth
locus of $X$, this is equivalent to saying that for any point (or some
point) $p\in
l$, the natural map $H^0(l,N_{l/X})\to N_{l/X}|p$ is not surjective.

\begin{lemma}\label{nonfreelines}
  If $X$ is a  complete intersection of multidegree $(d_1, \dots,
  d_m)$ in $\PP^n$ such that  $d_1+\dots+d_m \leq n-1$ then $X$ is
  covered by lines and  the set of non-free lines contained
  in the smooth locus of $X$ do not cover a dense subset of $X$.
\end{lemma}

\begin{proof}
  The first part is proved in \cite[Proposition 2.13]{debarre}. We
  just repeat the proof here. Let $p\in X$ and without loss of
  generality we may assume that $p=(1:0:\cdots :0)$ and $X$ is defined
  by $0=f_i=a_{i1}x_0^{d_i-1}+\cdots+ a_{id_i}$ for $1\leq i\leq m$,
  $a_{ij}$  homogeneous polynomials in $x_1,\ldots,x_n$. Since
  there are $\sum d_i$ of these $a_{ij}$s, they have a common
  non-trivial zero in $\PP^{n-1}=\{x_0=0\}$ since $n-1\geq \sum d_i$
  and then the line joining 
  $p$ and this point in $\PP^{n-1}$ is contained in $X$. Since $p$ was
  arbitrary, we see that $X$ is covered by 
  lines. 

 Let $\mathcal J$ be   the locally closed subscheme
 (with the reduced induced structure) of $R_1(X)$ (the set of lines in
 $X$) parametrizing
 non-free lines contained in the smooth locus of $X$.
 Let $\mathcal I \subset \mathcal J \times X$ denote the incidence
 correspondence. We wish to show that the projection $\cI\to X$ is not
 dominant. 
 Let  $([l], p)$ be a general point of  
$\mathcal I$. Then we have  $T_{R_1(X),[l]} \cong H^0(l,N_{l/X})$, and since $l$ is
contained in the smooth locus of $X$,  $T_{X,p}\to N_{l/X}|_p$ is
surjective. 

We have  a commutative diagram 
$$\xymatrix{
T_{\mathcal I, ([l],p)} \ar[rr] \ar[d] & & T_{X,p}
\ar@{->>}[d]  \\ 
T_{\mathcal J, [l]} \ar[r] & T_{R_1(X),[l]}=H^0(l,N_{l/X}) \ar[r] & N_{l/X}|_p 
}
$$
Since $l$ is not free, the map $H^0(l,N_{l/X}) \to N_{l/X}|_p$ is not
surjective, so $T_{\mathcal I, ([l],p)} \to  
T_{X,p}$ is not surjective and thus  $\cI\to X$ is not dominant. 
\end{proof}

\smallskip
\begin{lemma}\label{expected}
If $X \subset \PP^n $ is a smooth complete intersection of multidegree $(d_1, \dots, d_m)$ such that 
$d_1+\dots+d_m \leq n-1$, then $R_e(X)$ (and hence $\overline{\M}_{0,0,}(X,e)$) has at least one irreducible component of the 
expected dimension for every $e \geq 1$. 
\end{lemma}
\begin{proof}
We first show that smooth rational curves of degree $e$ in $X$ sweep
out a dense subset of $X$. From the previous lemma, $X$ is covered by
lines and any line passing through a general point is free.
Hence 
there is a chain of $e$ free lines in $X$ for every $e \geq 1$. By
\cite[Theorem 7.6]{kollar}, this chain of  
lines can be deformed to a smooth free rational curve $C$ of degree $e$ 
in $X$.  Since $C$ is free, its flat deformations in $X$ sweep out
$X$. 

If $R$ is a component of $R_e(X)$ such that the curves parametrized by its points sweep out a dense 
subset of $X$, then the  argument in the previous lemma
\ref{nonfreelines} shows that the normal 
bundle of a general curve $C$ parametrized by $R$ should be globally generated. 
We have
\begin{equation*}
\begin{split}
\dim T_{R_{e}(X),[C]} & = h^0(C, N_{C/X}) \\
& = \chi(N_{C/X}) + h^1(C,N_{C/X})\\
& = e(n+1-d)+(n-m-3).
\end{split}
\end{equation*}
On the other hand the dimension of $R$ is at least the expected dimension $e(n+1-d)+n-m-3$, so it should be of the expected dimension and 
smooth at $[C]$. 
\end{proof}
\medskip

\begin{proposition}\label{sweep}
If $X$ is a general hypersurface of degree $3 \leq d \leq n-1$ in $\PP^n$, then the non-free lines in $X$ sweep out a divisor in $X$.  
\end{proposition}
\begin{remark}
  If $d=1$, clearly this set is empty. Same is true if $d=2$. We have
  already seen that the codimension of this set is at least one in
  lemma \ref{nonfreelines}.
\end{remark}

\begin{proof}
Let $\HH$ be the projective space of hypersurfaces of degree $d$ in $\PP^n$.
Consider the subvariety $I \subset \PP^n \times \HH$ consisting of pairs $(p,[X])$ 
such that there is either a non-free line in $X$ through $p$ or a line in $X$ through $p$ which intersects the singular locus of $X$. 
Denote by $\pi_1$ and $\pi_2$ the projection maps from $I$ to $\PP^n$ and $\HH$.  
We show the fibers of $\pi_1$ are of dimension $=\dim \HH + n-2$. 

Since all the fibers of $\pi_1$ 
are isomorphic, we can assume $p = (1:0, \dots :0)$. A hypersurface $X$ which contains $p$ is given by an equation of the form $x_0^{d-1}f_1+\dots +f_d = 0$ where $f_i$ is homogenous of degree $i$ in $x_1,\dots,x_n$ for $1 \leq i \leq d$. The space of lines through $p$ in $X$, which we denote by 
$F_p(X)$, is isomorphic to the scheme $V(f_1,\dots, f_d)$ in $\PP^{n-1}$, so $\dim F_p(X) \geq n-d-1$. 
If $F_p(X)$ is singular at $[l]$, then we see that $([X],p) \in \pi_1^{-1}(p)$ as follows. 
Since 
$$T_{F_p(X),[l]} \cong H^0(l,N_{l/X}(-p)),$$ 
if $F_p(X)$ is singular at $[l]$, then 
$$
h^0(l,N_{l/X}(-p)) > \dim F_p(X) \geq n-d-1. 
$$
So either $l$ is not contained in the smooth locus of $X$, or 
it is contained in the smooth locus of $X$ and 
\begin{equation*}
\begin{split}
h^1(l,N_{l/X}(-p))  & = h^0(l,N_{l/X}(-p)) - \chi(N_{l/X}(-p)) \\
& = h^0(l,N_{l/X}(-p)) - (n-d-1) \\
&> 0,
\end{split}
\end{equation*}
so $l$ is non-free.
  
If $2 \leq d \leq n-1$, and if $f_i$ is a general homogeneous polynomial of degree $i$ in $x_1,\ldots,x_n$ for $2\leq i\leq d$, 
then the intersection $Y : = V(f_2,\dots,f_{d})$ is a smooth complete intersection 
subvariety of $\PP^{n-1}$ of dimension $n-d \geq 1$. By \cite[Proposition 3.1]{ein}, 
the dual variety of $Y$ in ${\PP^{n-1}}^{\vee}$ is a hypersurface, so
there is a codimension 1 subvariety of  
$H^0(\PP^{n-1}, \cO_{\PP^{n-1}}(1))$ consisting of forms $f_1$ such that $Y \cap \{f_1=0\}$ is singular. 
This shows that the space of tuples $(f_1, \dots, f_d)$ for which the scheme 
$V(f_1,\dots, f_d)$ is singular is of codimension 1 in the space of all tuples $(f_1, \dots, f_d)$.
So the fibers of $\pi_1$ over $p$ form a subvariety of codimension at most 1 in the space of all 
hypersurfaces which contain $p$, and $\dim I \geq \dim \HH + n-2$.

Consider now the map $\pi_2: I \to \HH$. Since the fibers of $\pi_2$ have dimension 
at most $n-1$, either $\pi_2$ is dominant or its image is of codimension 1 in $\mathcal H$.  We show 
the latter cannot happen.  For any  
hypersurface $X$, the space of lines which are contained in the smooth
locus of $X$ and are not free cannot sweep out a dense subset in $X$
by Lemma \ref{nonfreelines}, so if $\dim \pi_2^{-1}([X]) = n-1$, then   
the lines passing through the singular points of $X$ should sweep out $X$.  The locus of hypersurfaces which are singular at least along a curve is of codimension greater than 1 in $\mathcal H$, and so is the locus of hypersurfaces 
which are cones over hypersurfaces in $\PP^{n-1}$ when $d \geq 3$. Therefore $\pi_2$ is dominant, and $\dim I = \dim \mathcal H +n-2$, 
so a general fiber of $\pi_2$ has dimension $n-2$. 
 
\end{proof}

%\medskip

\begin{proof}[Proof of Theorem \ref{lines}]
Assume to the contrary that every smooth hypersurface of degree $d$ 
has a family of lines of dimension $(a_k+1)$ passing through one point such that 
for every line $l$ in the family, $h^1(l, N_{l/X}(-1)) \geq k$.  Let $\HH$ be the projective space of hypersurfaces of degree $d$ in $\PP^n$ and $\mathcal U$ the universal hypersurface over $\HH$. 
Let $F_{p,k}(X)$ be the subvariety of the Grassmannian of lines in $X$ passing though $p$ parametrizing lines $l$ 
with $h^1(l, N_{l/X}(-1)) \geq k$, and let  
$B$ be the closed subvariety of $\mathcal 
U \times \Gr(1,n)$ consisting of all the points $([X], p, [l])$ such that 
$ \dim F_{p,k}(X)$ at $[l]$ is  larger than $a_k$.

Denote by $\phi: B \to \mathcal U$ and $\pi_1: \mathcal U 
\to \HH$ the projection maps. By our assumption $\psi=\pi_1\circ \phi$ 
is dominant. We replace $B$ by an irreducible component of $B$ for
which $\psi$ is still dominant.
 By Proposition \ref{summary}, for a general point 
$(X,[p])$ in the image of $\phi$, there is a nonempty open subset
 $B_0$ of $\phi^{-1}([X],p)$ and a subspace  
$$
W_{X,p}  \subset H^0(\PP^n, \cO_{\PP^n}(d)(-p))
$$ 
of codimension at most $n-1$ such that for every  
$ b= ([X], p, [l]) \in B_0$ and  $w \in W_{X,p}$,  $w|_l$ can be lifted to a section of $N_{l/\PP^n}(-p)$ under the map obtained from the short exact sequence,
$$
0 \to N_{l/X}(-p) \to N_{l/\PP^n}(-p) \to \cO_l(d)(-p) \to 0.
$$

Suppose that $([X],p,[l])$ is a point of $B_0$. Let $\Gamma$ be a hyperplane in $\PP^n$ which 
does not pass through $p$. Choose a system of coordinates for $\PP^n$ so that 
$p = (1:0:\dots:0)$ and $\Gamma$ is given by $x_0=0$. Let $F$ be a
larger than expected  dimension irreducible component of
$F_{p,k}(X)$ to which $[l]$ belongs.  
The cone of lines parametrized by $F$ intersects $\Gamma$ along a subvariety $Y$ of dimension 
$\geq a_k+1$.

For every $i \geq 1$, multiplication by $x_0^{d-i}$ identifies $H^0(\Gamma, \cO_{\Gamma}(i))$ 
with a subspace of $H^0(\PP^n, \cO_{\PP^n}(d)(-p))$. 
Set 
$$
W_i := W_{X,p} \cap \; x_0^{d-i}H^0(\Gamma, \cO_{\Gamma}(i)). 
$$ 
Since the codimension of $W_{X,p}$ in $H^0(\PP^n, \cO_{\PP^n}(d)(-p))$ is at most $n-1$, 
the codimension of $W_i$ in $x_0^{d-i}H^0(\Gamma, \cO_{\Gamma}(i))$ is at most $n-1$.
By definition of $a_k$, for every $i$ with $k+1 \leq i \leq d$, $n-1 <
{{a_k+1+i} \choose {i}}.$  
Since $\dim Y \geq a_k+1$, by Lemma \ref{codimension} (i), 
there is $f_{i} = x_0^{d-i}g_i$ in $W_i$ 
such that $g_{i}$ does not vanish on $Y$. 
So for every $k+1 \leq i \leq d$, $f_{i} \in W_{X,p} \subset H^0(\PP^n, \cO_{\PP^n}(d)(-p))$, it 
vanishes to order exactly $i$ at $p$, and it is in the image of the map 
$$
\rho: H^0(l, N_{l/\PP^n}(-p)) \to H^0(l, \cO_l(d)(-p)).
$$ 
The argument at the end of the proof of Theorem \ref{main} shows that there is a section of $\cO_l(d)$ which 
vanishes to order 1 at $p$ and can be lifted to a section of $N_{l/\PP^n} (-p)$. Thus  
the dimension of the image of $\rho$ is at least $d-k+1$ and $h^1(l, N_{l/X}(-p)) \leq k-1$, which is a contradiction. 
\end{proof}

The proof of Theorem \ref{lines} yields stronger results if we know the dimension of the linear span of 
$Y$ defined in the proof of the theorem.  
The next proposition gives such a result when $k=1$. 

\begin{proposition}\label{linearspan}
Suppose that $X$ is a general hypersurface of degree $d$ in $\PP^n$ and $\Sigma$ is a cone of lines in $X$ over a curve $Y \subset \PP^{n-1}$. 
If the linear span of $Y$ has dimension $> (n-2)/2$, then a general line  parametrized by $Y$ is free. 
\end{proposition}
\begin{proof}
The proof is similar to that of Theorem \ref{lines} except that we apply part (ii) of Lemma \ref{codimension} to $Y$. Let $s$ be the dimension of the linear span of $Y$, and let 
$p$ be the vertex of the cone over $Y$. 
Since $X$ is general, there is a subspace $W_{X,p} \subset H^0(X, \cO_X(d)(-p))$ of codimension 
at most $n-1$ such that for every $w\in W$ and a general line $l$ parametrized by $Y$, $w|_l$ 
can be lifted to a section of $N_{l/\PP^n}(-p)$ under the map 
$$ 
\rho: H^0(l,N_{l/\PP^n}(-p)) \to H^0(l,\cO_l(d)(-p)).
$$ 
Let $W_i$ be defined as in the proof of Theorem \ref{lines}. By Lemma \ref{codimension} (ii),  if $2s+1 > n-1$, then for every $2 \leq i \leq d$, there is a section 
$f_i = x_0^{d-i}g_i \in W_{X,p}$ such that $g_i$ does not vanish on $Y$. So $f_i$ has a zero of 
order $i$ at $p$ and is contained in the image of $\rho$. 
The image of $\rho$ also contains a section which has a simple zero at 
$p$ by the same reasoning as in the proof of Theorem \ref{lines}, so $\rho$ is surjective and 
$H^1(l,N_{l/X}(-p))=0$.  

\end{proof}

Assume now that $X$ is a general complete intersection of multidegree $(d_1, \dots, d_m)$ in $\PP^n$, and put $d := d_1 + \dots + d_m$. 
It is likely that a similar strategy as in the proof of \cite[Theorem 2.1]{hrs} could be applied to show 
that if $\sum_i d_i \leq n-1$, the space of lines in $X$ through 
every point of $X$ has dimension equal to 
$n-\sum_i d_i-1$, but here we get a weaker result as an immediate corollary to a generalization of Theorem \ref{lines}.

\begin{proposition}\label{flat}
Let $X \subset \PP^n$ be a general complete intersection of multidegree $(d_1, \dots, d_m)$. Set $d = d_1 +\dots + d_m$.
\begin{itemize}
\item[(i)] 
If there is a family of dimension $r$ of non-free lines through a point $p$ in $X$, then ${{r+2}\choose{2}} \leq n-m$.  
\item[(ii)] If ${{n-d+2}\choose{2}} > n-m$, then the  evaluation map 
$ev: \overline{\M}_{0,1}(X,1) \to X$
is flat and has relative dimension $n-d-1$.  
\end{itemize}
\end{proposition}

\begin{proof}
(i) The proof is similar to the proof of Theorem \ref{main}. Let $s$ be the dimension of the largest family of non-free lines in a general complete 
intersection of multidegree $(d_1, \dots, d_m)$ which all pass through the 
same point. We show ${{s+2}\choose{2}} \leq n-m$. Assume to the contrary that the inequality does not hold. 
Let $\mathcal H$ denote the variety parametrizing complete intersections of multidegree $(d_1, \dots, d_m)$ in 
$\PP^n$, and let $\mathcal U$ denote 
the universal family over $\mathcal H$. Let $B \subset \mathcal U \times \Gr(1,n)$ 
be the subvariety parametrizing triples $([X],p,[l])$ such that $l$
belongs to a $s$-dimensional family  
of  lines through $p$ in $X$, and let $([X],p)$ be a general 
point in the image of the projection map $\phi: B \to \mathcal U$. 

The existence of a subspace 
$W_{X,p} \subset H^0(X,\cO_X(-p))$ of codimension at most $n-m$ as in the proof of Theorem \ref{lines} and the assumption that   ${{s+2}\choose{2}} > n-m$ 
show that if $([X],p,l)$ is a general point in $\phi^{-1}([X],p)$, then for every 
$1 \leq j \leq m$, and every $k \geq 2$, there is a global section $f_{j,k}$ of 
$N_{X/\PP^n}=\oplus_{i=1}^m\cO_{l}(d_i) $ such that 

\begin{itemize}
\item[(1)] The $i$-th component of $f_{j,k}$ is zero for $i \neq j$, 
\item[(2)] The $j$-th component of $f_{j,k}$ has a zero of order equal to $k$ at $p$,
\item[(3)] $f_{j,k}$ can be lifted to a section of $N_{l/\PP^n}(-p)$ via the map obtained from the exact sequence:
$$
0 \to N_{l/X}(-p) \to N_{l/\PP^n}(-p) \to \oplus_{i=1}^m\cO_{l}(d_i)(-p) \to 0.
$$  
\end{itemize}
On the other hand, the surjectivity of the map of sheaves $\cO_l(1)^{n+1} \to \oplus_{i=1}^m\cO_{l}(d_i)$ implies that for each $1 \leq j \leq m$, there is a 
global section $f_j$ of  $\oplus_{i=1}^m\cO_{l}(d_i)$ whose $j$-th component has a simple zero, 
and whose other components have zeros of order at least 2 at $p$,  
such that $f_j$ is in the image of the map 
$$H^0(l, \cO_{l}(1)^{n+1}(-p)) \to  H^0(l,\oplus_{i=1}^m\cO_{l}(d_i)(-p)) \subset H^0(l,  \oplus_{i=1}^m\cO_{l}(d_i)) .$$
This shows that the map $H^0(l, N_{l/X}(-p)) \to  H^0(l,\oplus_{i=1}^m\cO_{l}(d_i)(-p))$ is surjective, 
a contradiction. 

(ii) As  was shown in the proof of 
Theorem \ref{main}, to prove the flatness of $ev$, it suffices 
to show that the fibers of $ev$ have constant dimension $n-d-1$. But 
the fibers of $ev$ have dimension 
at least $n-d-1$, and if there is an irreducible component $\M$ of $ev^{-1}(p)$ whose dimension is larger than $n-d-1$, then 
every line parametrized by $\M$ should be non-free. This is not possible by part (i). 
\end{proof}
\bigskip

%%%%%%%%%%%%%%%%%%%%%%%%%%%%%%%%%%%%%%%%%%%%%%%%
\section{Dimension and irreducibility of $\overline{\M}_{0,0}(X,e)$}
%%%%%%%%%%%%%%%%%%%%%%%%%%%%%%%%%%%%%%%%%%%%%%%%

Suppose that $X$ is a complete intersection of multidegree $(d_1, \dots, d_m)$ in $\PP^n$, and set  $d : = d_1 + \dots + d_m$. If $d \leq n$, then 
the {\em threshold degree} of $X$ is defined to be 
$$E(X) : = \left\lfloor \frac{n+1}{n+1-d} \right\rfloor.$$ 
%By \cite[Corollary 5.5]{hrs}, if the evaluation map is flat and has relative dimension $e(n+1-d)-2$ for every $1 \leq e \leq E(X)$, then it is 
%flat and of relative dimension $e(n+1-d)-2$ for every $e \geq 1$. 

\begin{theorem}\label{last}
Let $X \subset \PP^n$ be a general complete intersection of multidegree $(d_1,\dots, d_m)$. If $d < 2n/3, n \geq 20$, then 
the evaluation map $ev: \overline{\M}_{0,1}(X,e) \to X$ is flat and has relative dimension 
$e(n+1-d)-2$ for every $e \geq 1$.
\end{theorem}

\begin{proof}
By \cite[Corollary 5.5]{hrs}, if the evaluation map $ev: \overline{\M}_{0,1}(X,e) \to X$ is flat of relative dimension 
$e(n+1-d)-2$ for every $1 \leq e \leq E(X)$, then it is 
flat of relative dimension $e(n+1-d)-2$ for every $e \geq 1$. 
If $d < \frac{2n}{3}$, then  $E(X) \leq 2$, so to prove the statement it is enough to prove it 
for $e=1,2$. If $e=1$,  our assumptions on $n$ and $d$ imply that ${{n-d+2}\choose{2}} > n-m$, so by Proposition \ref{flat}, $ev$ 
is flat of relative dimension $n-d-1$. 
If $e=2$, and $n,d$ satisfy the given inequalities, then $n-1<  {{n-d} \choose {2}}$, 
so Theorem \ref{main} shows that $ev$ is flat of the expected relative dimension when $X$ is a hypersurface. The same proof 
can be extended to the case of general complete intersections with $n-m<  {{n-d} \choose {2}}$. 
\end{proof}

\begin{theorem}
With the same assumptions as in Theorem \ref{last}, 
$\overline{\M}_{0,0}(X,e)$ is an irreducible complete intersection stack of dimension $e(n+1-d)+n-m-3$ for every $e \geq 1$.
\end{theorem}
\begin{proof}
By the previous theorem, $$\dim \overline{\M}_{0,0}(X,e) = \dim \overline{\M}_{0,1}(X,e)-1 = e(n+1-d)+n-m-3.$$ The stack  
$\overline{\M}_{0,0}(X,e)$ is the zero locus of a section of a locally free sheaf of rank $de+m$ over the smooth stack $\overline{\M}_{0,0}(\PP^n,e)$
 (see Section 2.1). Since $\dim \overline{\M}_{0,0}(\PP^n,e)= (e+1)(n+1)-4$, and since $\overline{\M}_{0,0}(X,e)$ 
has the expected dimension $= \dim  \overline{\M}_{0,0}(\PP^n,e) - (de+m)$, it is a local complete intersection stack.

Next we prove that $\overline{\M}_{0,0}(X,e)$ is irreducible. By \cite[Corollary 6.9]{hrs}, if $X$ is a smooth complete intersection, then $\overline{\M}_{0,0}(X,e)$ is irreducible for every $e \geq 1$ if 
 \begin{itemize}
 \item[(i)] The evaluation map $ev: \overline{\M}_{0,1}(X,e) \to X$ is flat of relative dimension $e(n+1-d)-2$ for every $e \geq 1$.
 \item[(ii)] General fibers of $ev$ are irreducible. 
 \item[(iii)] There is a free line in $X$.
 \item[(iv)] $\M_{0,0}(X,e)$ is irreducible for every $1 \leq e \leq E(X)$ where $\M_{0,0}(X,e)$ denotes the stack of stable maps 
 of degree $e$ with irreducible domains. 
 \end{itemize}
 
 By Theorem \ref{last} the first property is satisfied,  and property (iii)  
holds for every smooth complete intersection which is covered by lines, that is every smooth 
complete intersection with $d \leq n-1$. 
By Corollary 
\ref{normal}, for every line $l$ in $X$, $H^1(l,N_{l/X})=0$, so
$\overline{\M}_{0,0}(X,1)$  
and hence $\overline{\M}_{0,1}(X,1)$ are smooth. Therefore, by generic smoothness, a general fiber of 
$ev: \overline{\M}_{0,1}(X,1) \to X$ is smooth. Since every fiber of this map has the expected dimension $n-d-1$, it is a complete intersection 
of dimension $\geq 1$ in $\PP^{n-1}$, so it is also connected and therefore irreducible.

To show property (iv) holds, we need to show  $\M_{0,0}(X,e)$ is irreducible for $e=1,2$. $\M_{0,0}(X,1)$ is the space of lines in $X$ 
which is irreducible by (\cite[V.4.3]{kollar}) when $X$ is a smooth hypersurface of degree $\leq 2n-4$ in $\PP^n$, $n\geq 4$. 
The same proof can be generalized to the case of complete intersections. The irreducibility of the space of lines in general complete intersections 
with $d \leq n-1$ is also proved in \cite[Corollary 4.5]{jason-johan}. 

It is proved in \cite{deland} that $\M_{0,0}(X,2)$ is irreducible for 
a general hypersurface of degree $\leq n-2$ in $\PP^n$. Let us explain how one can generalizes the same argument to 
the case of general complete intersections with $d \leq n-2$. Note that since the dimension of $\M_{0,0}(X,2)$ is $3n-2d-m-1$, and since 
the space of lines passing through any point of $X$ has dimension $n-d-1$, the space of reducible conics in $X$ has dimension 
$$ \dim \M_{0,0}(X,1) + 1 + n-d-1 = 3n-2d-m-2 < \dim \M_{0,0}(X,2),$$
and the space of double covers of lines in $X$ has dimension
$$ \dim \M_{0,0}(X, 1) + 2 = 2n-d-m < 3n-2d-m-1,$$
every irreducible component of $\M_{0,0}(X,2)$ contains an open subscheme parametrizing smooth embedded 
conics in X. Therefore, to prove $\M_{0,0}(X,2)$ is irreducible, it is enough to show that $\Hilb_{2t+1}(X)$ is irreducible. To this end, 
let $I \subset \Hilb_{2t+1}(\PP^n) \times \mathcal H$ be the incidence correspondence parametrizing 
pairs $([C], [X])$ such that $C$ is a conic in $X$, and let 
$\pi_1: I \to  \Hilb_{2t+1}(\PP^n)$ and $\pi_2: I \to \mathcal H$ denote the two projection maps. Since $\Hilb_{2t+1}(\PP^n)$ is smooth and irreducible, and since 
the fibers of $\pi_1$ are product of projective spaces, $I$ is smooth and irreducible.  
%For $[X] \in \mathcal H$, $\pi_2^{-1}([X]) = \Hilb_{2t+1}(X)$.

Let $J$ be the closed subscheme of $I$ parametrizing pairs $([C],[X])$ such that $C$ is a non-reduced conic, so the support of $C$ is a line in $X$.
Then $J$ is smooth and irreducible since $J$ maps to the Grassmannian of lines in $\PP^n$ and the fibers are smooth and irreducible.  Let $\pi_2': 
J \to \mathcal H$ be the projection map. Note that for any smooth complete intersection $X$ and $l \subset X$, the space of non-reduced 
conics in $X$ whose support is $l$ can be identified with
$\PP(H^0(l,N_{l/X}(-1)))$. If $[X] \in \mathcal H$ is general, then the space of lines in $X$ is irreducible, thus  
the fiber of $\pi'_2$ over $[X]$ is connected. By generic smoothness, $\pi_2'^{-1}([X])$ is smooth and therefore irreducible.  
%If $[X] \in \mathcal H$ is general, then the space of lines in $X$ is irreducible, thus  
%the fiber of $\pi'_2$ over $[X]$ is connected and therefore irreducible.

%as well. Since $X$ is covered by lines, and since the space of lines in $X$ is irreducible, if $l$ is a general 
%line in $X$, then $l$ is free. Therefore $\overline{\M}_{0,0}(X,2)$ is smooth at double covers of general lines in $X$. 
By \cite[Lemma 3.2]{jason-johan}, if $i:N \to M$ and $e: M \to Y$ are morphisms of irreducible schemes and 
$i$ maps the generic point of $N$ to a normal point of $M$, then $e$ has 
irreducible general fibers provided that $e \circ i$ is dominant with irreducible general fibers. We apply this result to $N=J$, $M=I$, $Y = \mathcal H$, $i=$ the inclusion map, and $e=\pi_2$. Since $d \leq n-2, h^0(l, N_{l/X}(-1)) \geq 1$ for any smooth $X$ parametrized by $\mathcal H$ and any line $l \subset X$, so $e \circ i = \pi_2'$ is dominant and we have 
shown its general fibers are irreducible. Since $I$ is smooth, a general fiber of $\pi_2$ should be irreducible.

\end{proof}

%%%%%%%%%%%%%%%%%%%%%%%%%%%%%%%%%%%%%%%%%%%%%%%%%%%%%%%%%%%

\bigskip
\newcommand{\closer}{\vspace{-1.5ex}}

\end{document}